\documentclass[10pt]{amsart}
\usepackage{enumerate,amsmath,amssymb,latexsym,
amsfonts, amsthm, amscd, MnSymbol}

\theoremstyle{plain}

\newtheorem{theorem}{Theorem}

\numberwithin{equation}{section}

\begin{document}

\title {Nonelliptic functions from $F(\frac{1}{6}, \frac{5}{6} ; \frac{1}{2} ; \bullet)$ }

\date{}

\author[P.L. Robinson]{P.L. Robinson}

\address{Department of Mathematics \\ University of Florida \\ Gainesville FL 32611  USA }

\email[]{paulr@ufl.edu}

 \keywords{}

\begin{abstract}

As contributions to the Ramanujan theory of elliptic functions to alternative bases, Li-Chien Shen has developed families of elliptic functions from the hypergeometric functions $F(\tfrac{1}{3}, \tfrac{2}{3}; \tfrac{1}{2} ; \bullet)$ and $F(\tfrac{1}{4}, \tfrac{3}{4}; \tfrac{1}{2} ; \bullet)$. We apply his methods to the hypergeometric function $F(\tfrac{1}{6}, \tfrac{5}{6}; \tfrac{1}{2} ; \bullet)$. 

\end{abstract}

\maketitle

\medbreak

\section*{Introduction}

\medbreak 

Li-Chien Shen has shown in [4] and [6] how to extract elliptic functions from incomplete integrals of the hypergeometric functions $F(\tfrac{1}{3}, \tfrac{2}{3}; \tfrac{1}{2} ; \bullet)$ and $F(\tfrac{1}{4}, \tfrac{3}{4}; \tfrac{1}{2} ; \bullet)$. These contributions to the Ramanujan theory of elliptic functions to alternative bases are motivated by a reformulation of the corresponding classical theory. In the classical theory, if $0 < \kappa < 1$ then the relation 
$$u = \int_0^{\sin \phi} F(\tfrac{1}{2}, \tfrac{1}{2}; \tfrac{1}{2} ; \kappa^2 t^2) \, \frac{{\rm d} t}{\sqrt{1 - t^2}} $$
leads by inversion to $\phi = {\rm am} \, u$ along with the classical Jacobian elliptic functions ${\rm sn} \, u = \sin \phi$, ${\rm cn} \, u = \cos \phi$, and ${\rm dn} \, u = \frac{{\rm d} \phi}{{\rm d} u} = \sqrt{1 - \kappa^2 \sin^2 \phi}$. 

\medbreak 

In [4] the `classical' hypergeometric function is replaced by $F(\tfrac{1}{3}, \tfrac{2}{3}; \tfrac{1}{2} ; \bullet)$: this leads to an elliptic function ${\rm dn}_3 = \frac{{\rm d} \phi}{{\rm d} u}$; the analogous functions ${\rm cn}_3$ and ${\rm sn}_3$ are not elliptic, though their squares are. In [6] the `classical' hypergeometric function is replaced by $F(\tfrac{1}{4}, \tfrac{3}{4}; \tfrac{1}{2} ; \bullet)$: this leads to an elliptic function ${\rm dn}_2 = \sqrt{1 - \kappa^2 \sin^2 \phi}$; in this case, the analogues ${\rm cn}_2$ and ${\rm sn}_2^2$ are elliptic, though ${\rm sn}_2$ is not. In [2] and [3] we reviewed these constructions from a somewhat different perspective. 

\medbreak 

In [5] Shen considered the `signature 6' theory, but only in reference to the hypergeometric function $F(\tfrac{1}{6}, \tfrac{5}{6}; 1 ; \bullet)$. Here, we apply the methods of [4] and [6] to incomplete integrals of $F(\tfrac{1}{6}, \tfrac{5}{6}; \tfrac{1}{2} ; \bullet)$. Our findings in this case are essentially negative. The corresponding analogues of ${\rm sn}$ and ${\rm cn}$ are not elliptic; neither are their squares (nor their ratios). The analogue of the third Jacobi function in the form ${\rm dn} = \sqrt{1 - \kappa^2 \sin^2 \phi}$ is not elliptic; neither is its square. We leave open the question whether the analogue of the third Jacobi function in the form ${\rm dn} \, u = \frac{{\rm d} \phi}{{\rm d} u}$ is elliptic.
\medbreak 

\section*{The nonelliptic functions} 

\medbreak 

Let $0 < \kappa < 1$ and let $\lambda: = \sqrt{1 - \kappa^2} \in (0, 1)$. Write 
$$u = \int_0^{\sin \phi} F(\tfrac{1}{6}, \tfrac{5}{6}; \tfrac{1}{2} ; \kappa^2 t^2) \, \frac{{\rm d} t}{\sqrt{1 - t^2}}.$$
In a neighbourhood of the origin as a fixed point, the assignment $\phi \mapsto u$ inverts to $u \mapsto \phi$; further, there exists a function $\psi$ such that $\psi(0) = 0$ and $\sin \psi = \kappa \sin \phi$. 

\medbreak 

In these terms, we introduce functions $s, c, d, \delta$ as follows: 
$$s = \sin \phi, \; \; c = \cos \phi, \; \; d = \cos \psi, \; \; \delta = \phi\,'.$$ 
Originally, these functions are defined on a small disc about $0$; we wish to decide whether any of them (and functions related to them) admit elliptic extensions. 

\medbreak 

It is convenient to note that the functions $s, c$ and $d$ satisfy the same quadratic relations that are satisfied by the Jacobian elliptic functions ${\rm sn}$, ${\rm cn}$ and ${\rm dn}$ of modulus $\kappa$: namely, 
$$c^2 + s^2 = 1 \; \; {\rm and} \; \; d^2 + \kappa^2 s^2 = 1,$$
the second of these because 
$$1 - d^2 = 1 - \cos^2 \psi = \sin^2 \psi = \kappa^2 \sin^2 \phi = \kappa^2 s^2.$$
It is also convenient to recall the standard hypergeometric identity 
$$F(\tfrac{1}{6}, \tfrac{5}{6}; \tfrac{1}{2} ; \sin^2 z) = \frac{\cos \tfrac{2}{3} z}{\cos z}.$$

\medbreak 

We now have the following formula for $\delta = \phi\,'$. 

\medbreak 

\begin{theorem} \label{delta} 
$$\delta = \frac{\cos  \psi}{\cos \tfrac{2}{3} \psi}.$$
\end{theorem} 

\begin{proof} 
From the definition of $\phi \mapsto u$ we derive 
$$\frac{{\rm d} u}{{\rm d} \phi} = F(\tfrac{1}{6}, \tfrac{5}{6}; \tfrac{1}{2} ; \kappa^2 \sin^2 \phi) = F(\tfrac{1}{6}, \tfrac{5}{6}; \tfrac{1}{2} ; \sin^2 \psi)$$
whence by inversion and the hypergeometric identity recalled above 
$$\phi\,' = \frac{1}{F(\tfrac{1}{6}, \tfrac{5}{6}; \tfrac{1}{2} ; \sin^2 \psi)} = \frac{\cos  \psi}{\cos \tfrac{2}{3} \psi}.$$
\end{proof} 

\medbreak 

This result prompts us to introduce a new function (not operator!) $\partial$ by 
$$\partial = \cos \tfrac{2}{3} \psi$$
in terms of which the identity in Theorem \ref{delta} may be rewritten as 
$$\delta = \frac{d}{\partial}.$$ 

\medbreak 

The functions $d$ and $\delta$ are quite simply related. 

\medbreak 

\begin{theorem} \label{ratio}
$$\delta^3 - 2 d^2 \delta^3 - 3 d \delta^2 + 4 d^3 = 0.$$
\end{theorem} 

\begin{proof} 
By trigonometric duplication and triplication, 
$$2 \cos^2 \psi - 1 = \cos 2 \psi = 4 \cos^3  \tfrac{2}{3} \psi - 3 \cos  \tfrac{2}{3} \psi$$ 
whence Theorem \ref{delta} yields 
$$2 d^2 - 1 = 4 \partial^3 - 3 \partial = 4 (\tfrac{d}{\delta})^3 - 3 (\tfrac{d}{\delta})$$
and therefore 
$$(2 d^2 - 1) \delta^3 = 4 d^3 - 3 d \delta^2.$$
\end{proof} 

\medbreak 

As a step towards deciding whether or not the function $d$ admits an elliptic extension, we present the following differential equation. 

\medbreak 

\begin{theorem} \label{DEdelta}
$$d^2 \, (d\,')^2 = \delta^2 \,(1 - d^2) \, (d^2 - \lambda^2).$$ 
\end{theorem} 

\begin{proof} 
From the definition $\sin \psi = \kappa \sin \phi$ there follows 
$$\psi \,' \cos \psi = \kappa \phi\,' \cos \phi$$
so that 
$$\psi\,' = \kappa \cos \phi \, \frac{\phi\,'}{\cos \psi} = \kappa \cos \phi \, \frac{\delta}{d}$$
while from the definition $d = \cos \psi$ there follows 
$$d\,' = - \psi\,' \, \sin \psi = - \sin \psi \, \kappa \cos \phi \, \frac{\delta}{d}$$
so that 
$$(d \, d\,')^2 = \sin^2 \psi \, (\kappa^2 - \kappa^2 \sin^2 \phi) \, \delta^2 = (1 - \cos^2 \psi) \, (\kappa^2 - \sin^2 \psi) \, \delta^2 = (1 - d^2) \, (d^2 - \lambda^2) \, \delta^2$$
as announced. 

\end{proof} 

\medbreak 

This differential equation involves not only $d$ but also $\delta$. We may remove the additional dependence on $\delta$ by substitution of 
$$\frac{d^2}{\delta^2} = \frac{(1 - d^2) (d^2 - \lambda^2)}{(d\,')^2}$$ 
into the squared identity 
$$(2 d^2 - 1)^2 = \big[ 4 (\tfrac{d}{\delta})^2 - 3 \big]^2 (\tfrac{d}{\delta})^2$$
coming from Theorem \ref{ratio}. After elementary rearrangement, the result is the following differential equation satisfied by $d$ alone: 
$$(2 d^2 - 1)^2 (d\,')^6 = (1 - d^2) (d^2 - \lambda^2) \big[ 4 (1 - d^2) (d^2 - \lambda^2) - 3 (d\,')^2 \big]^2.$$

\bigbreak 

Alternatively, we may exchange the extra dependence on $\delta$ for an extra dependence on $\partial$, with the following result. 

\medbreak 

\begin{theorem} \label{DEpartial}
$$\partial^2 \, (d\,')^2 = (1 - d^2) \, (d^2 - \lambda^2).$$ 
\end{theorem} 

\begin{proof} 
Simply replace $\delta$ in Theorem \ref{DEdelta} by the ratio $d / \partial$ according to Theorem \ref{delta}. 
\end{proof} 

\medbreak 

We remark that the functions $c$ and $s$ similarly satisfy the differential equations
$$\partial ^2 \, (c\,')^2 = (1 - c^2) \, (\lambda^2 + \kappa^2 c^2)$$
and 
$$\partial^2 \, (s\,')^2 = (1 - s^2) \, (1 - \kappa^2 s^2).$$ 
Aside from the multiplier $\partial^2$, these are the familiar differential equations  
$$({\rm sn}\,')^2 = (1 - {\rm sn}^2) \, (1 - k^2 {\rm sn}^2)$$
$$({\rm cn}\,')^2 = (1 - {\rm cn}^2) \, (\ell^2 + k^2 {\rm cn}^2)$$
and 
$$({\rm dn}\,')^2 = (1 - {\rm dn}^2) \, ({\rm dn}^2 - \ell^2)$$ 
satisfied by the classical Jacobian elliptic functions ${\rm sn}$, ${\rm cn}$ and ${\rm dn}$ having modulus $k$ and complementary modulus $\ell$. 

\medbreak 

Going a little further, let us introduce the squared functions 
$$S = s^2, \; \; C = c^2 \; \; {\rm and} \; \; D = d^2.$$ 

\medbreak 

\begin{theorem} \label{squares} 
The squares $S, C$ and $D$ satisfy the differential equations 
$$\partial^2 \, (S\,')^2 = 4 S \, (1 - S) \, (1 - \kappa^2 S)$$
$$\partial^2 \, (C\,')^2 = 4 C \, (1 - C) \, (\lambda^2 + \kappa^2 C)$$ 
and 
$$\partial^2 \, (D\,')^2 = 4 D \, (1 - D) \, (D - \lambda^2).$$
\end{theorem} 

\begin{proof} 
The last of the equations follows directly from Theorem \ref{DEpartial} on account of the fact that $D\,' = 2 d d\,'$; the first pair of equations follows similarly from the pair displayed immediately after the same Theorem. 
\end{proof} 

\medbreak 

The reason for our passage from Theorem \ref{DEdelta} to Theorem \ref{DEpartial} is that the function $\partial$ satisfies a manageable differential equation. 

\medbreak 

\begin{theorem} \label{partial}
$$9 \, (\partial\,')^2 = 2 \, (1 - \partial^2) \, (4 \partial^3 - 3 \partial + 1 - 2 \lambda^2) / \partial^2.$$ 
\end{theorem} 

\begin{proof} 
Essentially along the same lines as the proof of Theorem \ref{DEdelta}. Note that 
$$\partial \,' = - \tfrac{2}{3} \, \psi\,' \, \sin \tfrac{2}{3} \psi = - \tfrac{2}{3} \, \sin \tfrac{2}{3} \psi \, \kappa \, \cos \phi / \partial$$
from which deduce that 
$$(\partial \partial \,')^2 = \tfrac{4}{9} \, \sin^2 \tfrac{2}{3} \psi \, \kappa^2 \, \cos^2 \phi = \tfrac{4}{9} \, (1 - \cos^2 \tfrac{2}{3} \psi) \, (\kappa^2 - \sin^2 \psi)$$
and finish the proof by noting that 
$$\kappa^2 - \sin^2 \psi = \cos^2 \psi - \lambda^2 = \tfrac{1}{2} \, (4 \partial^3 - 3 \partial + 1) - \lambda^2$$
by reference to the proof of Theorem \ref{ratio}. 
\end{proof} 

\medbreak 

We may draw useful inferences from the fact that $\partial$ satisfies this differential equation. If we are willing to import a celebrated generalization of the classical Malmquist theorem due to Yosida, we may deduce at once that the differential equation of Theorem \ref{partial} has no transcendental meromorphic solutions (in the plane) because the right-hand side is not a polynomial in $\partial$; see Section 6 of Chapter 4 on the Riccati equation in [1]. Our purposes are served by more modest inferences with more elementary justifications. 

\medbreak 

As a first inference, we have the following. 

\medbreak 

\begin{theorem} \label{partialno}
A (nonconstant) meromorphic solution $\partial$ of the differential equation 
$$9 \, \partial^2 \, (\partial\,')^2 = 2 \, (1 - \partial^2) \, (4 \partial^3 - 3 \partial + 1 - 2 \lambda^2)$$ 
can have no zeros. 
\end{theorem} 

\begin{proof} 
Deny: let $\partial$ be zero at some point $a$. Evaluation of the differential equation at $a$ forces $1 - 2 \lambda^2 = 0$. The differential equation then becomes 
$$9 \, \partial \, (\partial\,')^2 = 2 \, (1 - \partial^2) \, (4 \partial^2 - 3)$$ 
after an overall $\partial$ is cancelled. Finally, evaluation of this reduced differential equation at $a$ reveals the absurdity $0 = - 6$. 
\end{proof} 

\medbreak 

A further inference concerns the square (function, not operator) 
$$\nabla : = \partial^2.$$ 

\medbreak 

\begin{theorem} \label{DEnabla}
The function $\nabla = \partial^2$ satisfies 
$$\big[ \frac{9}{8} \, (\nabla\,')^2 + \Lambda \, (\nabla - 1) \big]^2 = \nabla \, (\nabla - 1)^2 \, (4 \nabla - 3)^2$$
where 
$$\Lambda = 1 - 2 \lambda^2.$$ 
\end{theorem} 

\begin{proof} 
From $\nabla = \partial^2$ it follows that $\nabla\,' = 2 \partial \, \partial\,'$ whence Theorem \ref{partial} yields 
$$9 \, (\nabla\,')^2 = 8 \, (1 - \nabla) \, \big[(4 \nabla - 3) \, \partial + 1 - 2 \lambda^2 \big]$$
or 
$$\frac{9}{8} \, (\nabla\,')^2 + \Lambda \, (\nabla - 1) = (1 - \nabla) \, (4 \nabla - 3) \, \partial$$
and squaring concludes the argument. 
\end{proof} 

\medbreak 

Theorem \ref{partialno} now improves as follows. 

\medbreak 

\begin{theorem} \label{nablano}
A (nonconstant) meromorphic solution $\nabla$ of the differential equation 
$$\big[ \frac{9}{8} \, (\nabla\,')^2 + \Lambda \, (\nabla - 1) \big]^2 = \nabla \, (\nabla - 1)^2 \, (4 \nabla - 3)^2$$ 
can have no zeros. 
\end{theorem} 

\begin{proof} 
Let us write $q$ for the quintic given by 
$$q(z) = z \, (z - 1)^2 \, (4 z - 3)^2$$
so that 
$$q(0) = 0 \; \; {\rm and} \; \; q'(0) = 9.$$
Let $\nabla$ be a meromorphic solution of the given differential equation 
$$\big[ \frac{9}{8} \, (\nabla\,')^2 + \Lambda \, (\nabla - 1) \big]^2 = q \circ \nabla$$
and differentiate 
$$2 \, \big[ \frac{9}{8} \, (\nabla\,')^2 + \Lambda \, (\nabla - 1) \big] \, (\frac{9}{4} \, \nabla\,' \, \nabla\,'' + \Lambda \, \nabla\,') = (q \,' \circ \nabla) \, \nabla\,'$$ 
whence by cancellation of $\nabla\,'$ we deduce that 
$$2 \, \big[ \frac{9}{8} \, (\nabla\,')^2 + \Lambda \, (\nabla - 1) \big] \, (\frac{9}{4}\, \nabla\,'' + \Lambda) = (q \,' \circ \nabla).$$ 
If possible, suppose that $\nabla (a) = 0$ and evaluate this last differential equation at $a$: the left-hand side is zero on account of the original differential equation and the fact that $q(0) = 0$; $q\,'(0) = 9$ tells us that the right-hand side is nonzero. Impossible. 
\end{proof} 

\medbreak 

We are now in a position to answer questions of ellipticity. We answer this question first for the function $\nabla$: when we say below that $\nabla$ is not elliptic, we of course mean that $\nabla$ does not admit an elliptic extension from the small disc about $0$ on which it is originally defined; likewise for the functions that we discuss subsequently. 

\medbreak 

\begin{theorem} \label{nablanotE}
The function $\nabla = \cos^2 \tfrac{2}{3} \psi$ is not elliptic. 
\end{theorem} 

\begin{proof} 
Because (nonconstant) elliptic functions have zeros, this is an immediate corollary of Theorem \ref{nablano}. 
\end{proof} 

\medbreak 

Of course, it follows that $\partial = \cos \tfrac{2}{3} \psi$ is also not elliptic. 

\medbreak 

\begin{theorem} \label{SCDnotE}
The functions $S = \sin^2 \phi, \; C = \cos^2 \phi$ and $D = \cos^2 \psi$ are not elliptic. 
\end{theorem} 

\begin{proof} 
Refer to Theorem \ref{squares}: the equation for $D = d^2$ may be reformulated as 
$$\nabla = 4 D \, (1 - D) \, (D - \lambda^2) / (D\,')^2;$$
accordingly, ellipticity of $D$ would force ellipticity upon $\nabla$ and thereby contradict Theorem \ref{nablanotE}. That $S = s^2$ and $C = c^2$ are not elliptic follows similarly or by virtue of the identities $d^2 + \kappa^2 s^2 = 1$ and $c^2 + s^2 = 1$. 
\end{proof} 

\medbreak 

Of course, $s, c$ and $d$ are also not elliptic. 

\medbreak 

We offer one more result of a similar nature, pertaining to $t = s/c$ and its square $T = t^2$. 

\medbreak 

\begin{theorem} \label{t}
$T = s^2 / c^2$ satisfies the differential equation 
$$\nabla (T\,')^2 = 4 T \, (1 + T) \, (1 + \lambda^2 T)$$ 
whence neither $T$ nor $t = s/c$ is elliptic. 
\end{theorem} 

\begin{proof} 
Differentiate $t$: as $s\,' = c \, \delta$ and $c\,' = -s \, \delta$ there follows
$$t\,' = (1 + t^2) \, \delta$$
on account of the identity $c^2 + s^2 = 1$, which of course also implies that $c^2 = 1/(1 + t^2)$. Square and use 
$$d^2 = 1 - \kappa^2 s^2 = \lambda^2 + \kappa^2 c^2 = \frac{1 + \lambda^2 t^2}{1 + t^2}$$
to deduce that 
$$(t\,')^2 = (1 + t^2)^2 \, \delta^2 = (1 + t^2)^2 d^2 /\partial^2 = (1 + t^2)(1+ \lambda^2 t^2)/\partial^2$$
and rearrange using $T\,' = 2 t \, t\,'$ to end the proof. 
\end{proof} 

\medbreak 

\bigbreak

\begin{center} 
{\small R}{\footnotesize EFERENCES}
\end{center} 
\medbreak 

[1] E. Hille, {\it Ordinary Differential Equations in the Complex Domain}, Wiley-Interscience (1976); Dover Publications (1997). 

\medbreak 

[2] P.L. Robinson, {\it Elliptic functions from $F(\frac{1}{3}, \frac{2}{3} ; \frac{1}{2} ; \bullet)$}, arXiv 1907.09938 (2019). 

\medbreak 

[3] P.L. Robinson, {\it Elliptic functions from $F(\tfrac{1}{4}, \tfrac{3}{4}; \tfrac{1}{2} ; \bullet)$}, arXiv 1908.01687 (2019). 

\medbreak 

[4] Li-Chien Shen, {\it On the theory of elliptic functions based on $_2F_1(\frac{1}{3}, \frac{2}{3} ; \frac{1}{2} ; z)$}, Transactions of the American Mathematical Society {\bf 357}  (2004) 2043-2058. 

\medbreak 

[5] Li-Chien Shen, {\it A note on Ramanujan's identities involving the hypergeometric function $F(\tfrac{1}{6}, \tfrac{5}{6}; 1 ; z)$}, Ramanujan Journal {\bf 30} (2013) 211-222. 

\medbreak 

[6] Li-Chien Shen, {\it On a theory of elliptic functions based on the incomplete integral of the hypergeometric function $_2 F_1 (\frac{1}{4}, \frac{3}{4} ; \frac{1}{2} ; z)$}, Ramanujan Journal {\bf 34} (2014) 209-225. 

\medbreak

\medbreak

\end{document}